\documentclass [11pt,oneside]{amsart}

\reversemarginpar \textwidth=14cm \textheight=19cm

\usepackage{amssymb} \usepackage{amsfonts} \usepackage{amsmath}
\usepackage{amsthm} \usepackage{epsfig} 
\usepackage{tikz}
\usepackage[cmtip,arrow]{xy}
\usepackage{amsmath,amssymb,enumerate,pb-diagram,pb-xy}

\usepackage[applemac]{inputenc}
\input xy
\xyoption{all}

\usepackage{caption}

\newtheorem{lemma}{Lemma}[section]
\newtheorem{teo}[lemma]{Theorem}
\newtheorem{prop}[lemma]{Proposition}
\newtheorem{cor}[lemma]{Corollary}

\theoremstyle{definition}
\newtheorem{defn}[lemma]{Definition}

\newtheorem{rem}[lemma]{Remark}

\newcommand{\matN}{\ensuremath {\mathbb{N}}}
\newcommand{\R} {\ensuremath {\mathbb{R}}}

\newcommand{\Z} {\ensuremath {\mathbb{Z}}}

\newcommand{\matH} {\ensuremath {\mathbb{H}}}

\newcommand{\vare} {\ensuremath{\varepsilon}}
\newcommand{\vol} {{\rm Vol}}
\newcommand{\vola} {{\rm Vol}_{\rm alg}}
\newcommand{\lamtil} {\widetilde{\Lambda}}
\newcommand{\lam}{\Lambda}
\newcommand{\mhat}{\widehat{M}}
\newcommand{\ext}{{\rm ext}}

\newcommand{\sm} {\ensuremath {{\rm sm}}}
\newcommand{\str} {\ensuremath {{\rm str}}}

\newcommand {\bb} {\partial}

\author{Roberto Frigerio}
\author{Cristina Pagliantini}

\address{Dipartimento di Matematica \\
Universit\`a di Pisa \\
Largo B.~Pontecorvo 5 \\
56127 Pisa, Italy}

\email{frigerio@dm.unipi.it, pagliantini@mail.dm.unipi.it}

\title[Simplicial volume of bounded hyperbolic manifolds]{The simplicial volume of hyperbolic manifolds\\ 
with geodesic boundary}



\thanks{}

\begin{document}

\begin{abstract}
Let $n\geq 3$, let $M$ be 
an orientable complete finite volume hyperbolic $n$-manifold
with compact (possibly empty) geodesic boundary, and let $\vol (M)$ and $\| M\|$
be the Riemannian volume and the simplicial volume of $M$.
A celebrated result by Gromov and Thurston 
states that if $\bb M=\emptyset$ then
$\vol (M)/ \|M\|=v_n$, where
$v_n$ is the volume of the 
regular ideal geodesic $n$-simplex in hyperbolic $n$-space.
On the contrary, Jungreis and Kuessner proved that if $\bb M\neq \emptyset$ 
then $\vol (M)/\|M\|<v_n$.  

We prove here that for every $\eta>0$ there exists $k>0$ (only depending on $\eta$ and $n$)
such that if $\vol (\partial M)/ \vol (M)\leq k$, then $\vol (M)/\|M\|\geq v_n-\eta$. 
As a consequence we show that for every $\eta>0$ there exists
a compact orientable hyperbolic $n$-manifold
$M$ with non-empty geodesic boundary such that $\vol (M)/\|M\|\geq v_n-\eta$.

Our argument also works in the case of empty boundary, thus providing a somewhat new proof
of the proportionality principle for non-compact finite-volume hyperbolic $n$-manifolds without boundary.
\end{abstract}

\maketitle

\section{Preliminaries and statements}\label{construction:section}
Let $X$ be a topological space, let $Y\subseteq X$ be a (possibly empty) subspace of $X$, and let
$R$ be a ring (in the present paper only the cases $R=\R$ and $R=\Z$ will be considered).
For $i\in\matN$ we denote  by $C_i(X;R)$ the module
of singular $i$-chains over $R$, \emph{i.e.}~the $R$-module
freely generated by the set $S_i (X)$ of singular $i$-simplices with values in $X$. 
The natural inclusion of $Y$ in $X$ induces an inclusion of $C_i (Y;R)$ into $C_i (X;R)$,
so it makes sense to define 
$C_i (X,Y;R)$ as the quotient space $C_i (X;R)/C_i (Y;R)$ (of course, if $Y=\emptyset$ we get $C_i (X,Y;R)=
C_i (X,R)$). The usual differential of the complex $C_\ast (X;R)$ defines
a differential $d_\ast\colon C_{\ast} (X,Y;R)\to C_{\ast -1} (X,Y;R)$. The homology of the resulting
complex is the usual relative singular homology of the topological pair $(X,Y)$, and will be denoted
by $H_\ast (X,Y;R)$.

In what follows, we will denote simply by $C_i (X)$, $C_i (X,Y)$ respectively the modules
$C_i (X;\R)$, $C_i (X,Y;\R)$.
The $\R$-vector space $C_i (X,Y)$
can be endowed with the following natural $L^1$-norm:  
if $\alpha\in C_i (X,Y)$, then
$$
\|\alpha\|=\|\alpha\|_1 =\inf \left\{\sum_{\sigma\in S_i (X)} |a_\sigma|\, ,\, {\rm where}\ 
\alpha=\left[\sum_{\sigma\in S_i (X)} a_\sigma \sigma\right]\ {\rm in}\ C_i(X)/C_i (Y) \right\}.
$$ 
Such a norm
descends to a seminorm on $H_\ast (X,Y)$, which is defined as follows:
if $[\alpha]\in H_i (X,Y)$, then
$$
\|[\alpha ]\|  =  \inf \{\|\beta \|,\, \beta\in C_i (X,Y),\, d\beta=0,\, [\beta ]=[\alpha ] \}
$$
(note that such a seminorm can be null on non-zero elements of $H_\ast (X,Y)$).

\subsection{Simplicial volume}
Throughout the whole paper, every manifold is assumed to be connected and orientable.
If $M$ is a compact $n$-manifold with (possibly empty) boundary $\bb M$, 
then we denote by $[M]_\mathbb{Z}$ a generator
of $H_n (M,\bb M; \mathbb{Z})\cong \mathbb{Z}$. Such a generator is usually known as the
\emph{fundamental class} of the pair $(M,\bb M)$.
The inclusion $\mathbb{Z} \hookrightarrow \R$ induces a map
$l\colon H_n (M,\bb M; \mathbb{Z})\hookrightarrow H_n (M,\bb M; \R)=H_n (M,\bb M)\cong\R$, and we
set $[M]_\R= l ([M]_\mathbb{Z})$. The following definition is due to Gromov~\cite{Gromov, Thurston}:

\begin{defn}
The \emph{simplicial volume} of $M$ is 
$\|M\|= \|[M]_\R\|$. 
\end{defn}

Since continuous maps induce norm non-increasing maps on singular chains and a homotopy equivalence
of pairs $f\colon (M,\bb M)\to (N,\bb N)$ between $n$-manifolds maps the fundamental class of 
$M$ into the fundamental class of $N$,
it is readily seen that the simplicial volume of a compact  manifold $M$
is a homotopy invariant of the pair $(M,\bb M)$.

As Gromov pointed out in his seminal work~\cite{Gromov}, 
even if it depends only on the homotopy type of a manifold, the simplicial volume is deeply related
to the geometric structures that a manifold can carry. For example, closed manifolds which support 
negatively curved Riemannian metrics have non-vanishing simplicial volume, while the simplicial 
volume of flat or spherical manifolds is null (see \emph{e.g.}~\cite{Gromov}). 
Even if several vanishing and non-vanishing results
for the simplicial volume
have been established (see \emph{e.g.}~\cite{Gromov, Lafont, Bucher2}), it is maybe worth mentioning
that, as far as the authors know, 
the exact value of non-null simplicial volumes is known at the moment only in the following 
cases:
a celebrated result by Gromov and 
Thurston (see Theorem~\ref{closed:teo} below) computes the simplicial volume of closed (and cusped)
hyperbolic manifolds, and the simplicial volume of the product of compact orientable surfaces
has been recently determined in~\cite{Bucher3}.

\subsection{The simplicial volume of hyperbolic manifolds}
Let $n\geq 3$. 
Throughout the whole paper, by \emph{hyperbolic $n$-manifold} we will mean 
a complete finite-volume 
hyperbolic $n$-manifold $M$ with compact (possibly empty)
geodesic boundary. 
We recall that if $M$ is cusped, \emph{i.e.}~if it is non-compact, 
then $M$ naturally compactifies to a manifold with boundary
$\overline{M}$ obtained by adding to $M$ a finite number of boundary
$(n-1)$-manifolds supporting 
a flat structure (see Subsection~\ref{N1:sub} below).

Let $v_n$ be the supremum of volumes of geodesic $n$-simplices
in hyperbolic $n$-space $\matH^n$. It is well-known~\cite{MM,MM2} that $v_n$ equals in fact the volume
of the geodesic regular ideal $n$-simplex. 
The following result is due to 
Thurston and Gromov~\cite{Thurston, Gromov}
(detailed proofs can be found in~\cite{BePe} for the closed case, and in~\cite{stefano,Kuessner}
for the cusped case):

\begin{teo}[Gromov, Thurston]\label{closed:teo}
Suppose $M$ is a hyperbolic $n$-manifold without boundary.
Then $\|\overline{M}\|\neq 0$ and 
$$
\frac{\vol (M)}{\|\overline{M}\|}=v_n.
$$
\end{teo} 

A different result holds for 
hyperbolic manifolds with \emph{non-empty} geodesic boundary:

\begin{teo}[\cite{Jung, Kuessner2}]\label{jung:teo}
Let $M$ be a hyperbolic $n$-manifold with non-empty
geodesic boundary. Then $\|\overline{M}\|\neq 0$ and
$$
\frac{\vol (M)}{\|\overline{M}\|}<v_n.
$$ 
\end{teo}

In this paper we show how 
to control the gap between ${\vol (M)}/\|\overline{M}\|$ and $v_n$
in terms of the ratio between the $(n-1)$-dimensional volume of $\partial M$ and the 
$n$-dimensional volume of $M$.
More precisely, in Section~\ref{main:sec} we prove the following:

\begin{teo}\label{main:teo}
Let $\eta>0$. Then there exists $k>0$ depending only on $\eta$ and $n$ such that
the following result holds:
if $M$ is a hyperbolic $n$-manifold with non-empty
geodesic boundary such that 
$$
\frac{\vol (\bb M)}{\vol (M)} \leq k,
$$
then 
$$
\frac{\vol (M)}{\|\overline{M}\|} \geq v_n - \eta.
$$
\end{teo}

It is not difficult to show that for every $n\geq 3$ there exist compact
hyperbolic $n$-manifolds
with non-empty disconnected geodesic boundary (see for example~\cite[Example 2.8.C]{GroPia}).
Let $M$ be one such manifold, choose one connected component $B_0$ of $\bb M$ and let
$M'$ be the manifold obtained by mirroring $M$ along $\bb M\setminus B_0$, so
$\bb M'$ is isometric to two copies of $B_0$. For $i\geq 1$, we inductively construct $M_i$ by setting
$M_1=M'$ and defining $M_{i+1}$ as the manifold obtained by isometrically gluing one component
of $\bb M_i$ to one component of $\bb M_1$. It is readily seen that $M_i$ is a compact hyperbolic
$n$-manifold with non-empty geodesic boundary such that $\vol (M_i)=2i \vol (M)$ and
$\vol (\bb M_i)=2\vol (B_0)$. We have therefore $\lim_{i\to\infty} \vol (\bb M_i)/\vol (M_i)=0$.
Together with our main theorem, this readily implies the following:

\begin{cor}\label{cor:cor}
For every $\eta>0$, a compact  hyperbolic $n$-manifold $M$ with non-empty
geodesic boundary exists such that
$$
v_n > \frac{\vol (M)}{\|\overline{M}\|} \geq v_n-\eta.
$$
\end{cor}
 
Finally, it is maybe worth mentioning that our proof of Theorem~\ref{main:teo} can be applied
word by word (with obvious simplifications) to 
hyperbolic manifolds \emph{without} boundary, thus providing a somewhat new proof of 
Theorem~\ref{closed:teo} in the case of non-compact manifolds.

\subsection{Strategy of the proof}\label{strategy}
Let $M$ be a hyperbolic $n$-manifold with possibly empty geodesic boundary.
Once a \emph{straightening} procedure is defined which allows to compute the simplicial
volume of $M$ only considering linear combinations of geodesic simplices, 
the inequality $\vol (M) / \|\overline{M}\| \leq v_n$ is easily established. In 
the case without boundary, in 
order to prove
the converse inequality (which fails in the case with boundary), 
one has to show that for every $\vare>0$ a cycle 
$\alpha_\vare \in C_n (\overline{M},\bb\overline{M})$ exists which represents the fundamental class
of $\overline{M}$ and is such that $\|\alpha\|\leq {\vol (M)}/{v_n}+\vare$. 
Such a cycle
is said to be \emph{$\vare$-efficient}, and it is not difficult to
show that a cycle is $\vare$-efficient if its simplices
all have hyperbolic volume close to $v_n$ (see \emph{e.g.}~\cite{BePe}).

In the same spirit we prove here that if $M$ has non-empty
geodesic boundary then the gap between ${\vol (M)}/\|\overline{M}\|$ and $v_n$ is
bounded from above
by the amount of simplices in any fundamental cycle of $\overline{M}$
whose volume is forced
to be far from $v_n$.
In order to obtain Theorem~\ref{main:teo} we then
show how to control such amount
of simplices in terms of the ratio between the $(n-1)$-dimensional volume of $\partial M$ and the 
$n$-dimensional volume of $M$.

More precisely, 
in Section~\ref{hyp:sec} we briefly describe 
some results about the geometry of hyperbolic manifolds with geodesic boundary. 
Section~\ref{smearing:sec}, which uses tools from~\cite[Section 4]{Loh-Sauer}, 
is devoted to the description of a discrete version
of Thurston's smearing construction which is very useful for exhibiting efficient
cycles. Finally in Section~\ref{main:sec} we provide the needed estimates
on the norm of such cycles 
thus concluding the proof of Theorem~\ref{main:teo}.

\section{Hyperbolic manifolds with geodesic boundary}\label{hyp:sec}
Let $n\geq 3$ and let $M$ be a hyperbolic $n$-manifold
with \emph{non-empty} geodesic boundary. 
This section is devoted to a brief description of the geometry
of $M$.

\subsection{Natural compactification}\label{N1:sub}  
Since $\bb M$ is compact, $M$ decomposes as the union of a compact smooth manifold with boundary 
$N\subseteq M$ with $\bb M\subseteq N$
and a finite number of cusps of the form $T_i\times [0,\infty)$,
$i=1,\ldots,r$,
where $T_i$ is a closed Euclidean $(n-1)$-manifold for every $i$
(see \emph{e.g.}~\cite{Kojima1, Kojima2, Frigerio}, where also the non-compact boundary
case is considered). Moreover, $N$ can be chosen in such a way that each cusp $T_i\times [0,\infty)$
is isometric to the quotient of a closed horoball in $\matH^n$ by a 
parabolic group of isometries.

Up to choosing ``deeper'' cusps, given
$\vare>0$ we may suppose that the volume of $M\setminus N$ is at most $\vare$, and 
if this is the case we denote $N$ by the symbol 
$M_\vare$.
We also denote by $\bb M_\vare$ the boundary of $M_\vare$ as a topological manifold
and we set ${\rm int} (M_\vare)=M_\vare\setminus \bb M_\vare$.
Observe that $\bb M_\vare$ is given by the union of $\bb M$
and the boundaries of the deleted cusps.
 
The description of $M$ just given implies that there exists a well-defined piecewise smooth
nearest point retraction 
$M\to M_{\vare}$ which maps $M\setminus {\rm int} (M_\vare)$ onto
$\bb M_\vare$. Moreover,
$M$ admits
a natural compactification $\overline{M}$ which is obtained by adding
a closed Euclidean $(n-1)$-manifold for each cusp and is homeomorphic to
$M_\vare$.

\subsection{Universal covering}\label{N2:sub}
Let $\pi:\widetilde{M}\to M$ be
the universal covering of $M$. By developing $\widetilde{M}$
in $\matH^n$ we can identify $\widetilde{M}$ with
a convex polyhedron of $\matH^n$ bounded by
a countable number of disjoint geodesic hyperplanes.
The group of the automorphisms of the covering
$\pi:\widetilde{M}\to M$ can be identified in a natural way
with 
a discrete torsion-free subgroup $\Gamma$ of 
$\mathrm{Isom}^+ (\widetilde{M})<\mathrm{Isom}^+(\matH^n)$ such that
$M\cong\widetilde{M}/\Gamma$.
Also recall that there exists an isomorphism
$\pi_1(M)\cong\Gamma$, which is canonical up to
conjugacy. With a slight abuse, from now on we refer to 
$\Gamma$ as to the fundamental
group of $M$. 

The covering $\pi\colon \widetilde{M}\to M$ extends to a covering 
$\matH^n\to \matH^n/\Gamma=\widehat{M}$, which will still be denoted by $\pi$.
Being the quotient of $\matH^n$ by a discrete torsion-free group of isometries,
$\widehat{M}$ is a complete (infinite volume) hyperbolic manifold without boundary. 
The inclusion $\widetilde{M}\hookrightarrow \matH^n$ induces an isometric inclusion
$M\hookrightarrow \widehat{M}$ which realizes $M$ as the \emph{convex core}
of $\widehat{M}$ (see \emph{e.g.}~\cite{Frigerio}). Therefore, 
there exists a well-defined piecewise smooth nearest point retraction
of $\widehat{M}$ onto $M$, which maps $\widehat{M}\setminus {\rm int} (M)$ onto $\partial M$.

Let $\ext (M_\vare)=\widehat{M}\setminus {\rm int} (M_\vare)$.
If $\vare$ 
is sufficiently small, 
by composing the retractions $\widehat{M}\to M$, $M\to M_\vare$ 
mentioned above we get a retraction
$p\colon \widehat{M}\to M_\vare$ such that $p(\ext (M_\vare))\subseteq
\bb M_\vare$. Such a map is piecewise smooth and induces a homotopy equivalence of pairs 
$p\colon (\widehat{M},\ext (M_\vare))\to (M_\vare, \bb M_\vare)$.

Since $M$ retracts to $M_\vare$ via a homotopy equivalence, the set
$\widetilde{M}_\vare=\pi^{-1} (M_\vare)\subseteq \widetilde{M}$ is simply connected, and provides
therefore the Riemannian universal covering of $M_\vare$. 
If $\ext (\widetilde{M}_\vare)=\matH^n \setminus \widetilde{M}_\vare$, then by construction 
$\ext (\widetilde{M}_\vare)$ is a $\Gamma$-invariant
disjoint union of closed half-spaces and closed horoballs. In particular,
every component of $\ext (\widetilde{M}_\vare)$ is  convex.

\section{Straightening and smearing}\label{smearing:sec}
Introduced by Thurston in~\cite{Thurston}, the smearing construction plays a fundamental r\^ole in
several proofs of Theorem~\ref{closed:teo} (see \emph{e.g.}~\cite{Thurston, Kuessner}).
Such a construction takes usually place in the setting of the so-called
\emph{measure homology}~\cite{Thurston, Zastrow, Loh}.
In order to make the proof of Theorem~\ref{main:teo} as self-contained as possible
and to get rid of some technicalities, we follow here some ideas described in~\cite{Loh-Sauer},
where a ``discrete version'' of the smearing construction is introduced.

\subsection{Straight simplices}\label{straight1:sub}
We now fix some notations we will be extensively using from now on.
For $i\in\matN$ we denote by $e_i$ the point $(0,0,\ldots,1,\ldots,0,0,\ldots)\in \R^{\matN}$
where the unique non-zero coefficient is at the $i$-th entry (entries are indexed by $\matN$,
so $(1,0,\ldots)=e_0$). 
We denote by $\Delta_p$ the standard $p$-simplex, 
\emph{i.e.}~the convex hull of $e_0,\ldots,e_p$, and we observe
that with these notations we have $\Delta_p\subseteq \Delta_{p+1}$.
If $\sigma\in S_p (X)$ is a singular simplex, we let $\bb_i \sigma \in S_{p-1} (X)$ be
the $i$-th face of $\sigma$. 

Let $k\in\matN$, and let $x_0,\ldots,x_k$ be points in $\matH^n$. We recall here
the well-known definition of \emph{straight} simplex $[x_0,\ldots, x_k]\in S_k (\matH^n)$
with vertices $x_0,\ldots,x_k$:
if $k=0$, then $[x_0]$ is the $0$-simplex with image $x_0$; if straight simplices have
been defined for every $h\leq k$, then $[x_0,\ldots,x_{k+1}]\colon \Delta_{k+1}\to\matH^n$
is determined by the following condition:
for every $z\in \Delta_k\subseteq \Delta_{k+1}$, the restriction
of $[x_0,\ldots,x_{k+1}]$ to the segment with endpoints
$z,e_{k+1}$ is a constant speed parameterization of
the geodesic joining $[x_0,\ldots,x_k] (z)$ to $x_{k+1}$
(the fact that $[x_0,\ldots,x_{k+1}]$ is well-defined and continuous is an obvious
consequence of the fact that 
any two given points 
in $\matH^n$ are joined by a unique geodesic, and hyperbolic geodesics
continuously depend on their endpoints). It is not difficult to show that
the image of a straight $k$-simplex coincides with the hyperbolic convex hull
of its vertices, and is therefore a (possibly degenerate) geodesic
$k$-simplex. 

Keeping notations from Subsections~\ref{N1:sub}, \ref{N2:sub},
if $\mhat$ is a hyperbolic $n$-manifold
with universal covering $\pi\colon \matH^n\to\mhat$ then 
we say that $\sigma\colon \Delta_k \to \mhat$ is \emph{straight} if it is obtained
by composing a straight simplex in $\matH^n$ with the covering projection
$\pi$.

The results stated in the following remarks will not be used
in this paper.

\begin{rem}
Let $\sigma\in S_k (\mhat)$, take  a lift $\widetilde{\sigma}\in S_k (\matH^n)$
and define $\str_k (\sigma)=\pi\circ [\widetilde{\sigma}(e_0),\ldots,\widetilde{\sigma}(e_k)]$.
Since isometries preserve geodesics we
have $\gamma\circ [x_0,\ldots,x_k]=[\gamma (x_0),\ldots, \gamma (x_k)]$
for every 
$\gamma \in \Gamma$, so $\str_k (\sigma)$ does not depend on the choice
of $\widetilde{\sigma}$.
It is well-known that $\str_\ast$ linearly extends to 
a chain endomorphism of $C_\ast (\mhat)$ which is algebraically homotopic to the identity.
\end{rem}

\begin{rem}
Straight simplices may be defined in the much more general setting of non-positively curved
complete Riemannian manifolds. Almost all the properties of straight simplices
(and of the associated straightening procedure)
described above also hold in this wider context. It is maybe worth mentioning, however, that
in the simply connected non-constant curvature case, straight simplices need not be convex.
\end{rem}

\subsection{Haar measure}\label{haar:sub}
 Let $G$ be a locally compact Hausdorff topological group and $\mathcal{A}$ be the $\sigma$-algebra 
of the Borelian subsets of $G$. A measure $\mu$ on 
$\mathcal{A}$ is called \emph{regular} if for each $A\in \mathcal{A}$
$$
\begin{array}{rcl}
\mu(A)&=&\sup \{\mu(K)\;|\;K\; \mbox{compact set},\;K\subset A\}\\
\mu(A)&=&\inf \{\mu(U)\;|\;U\; \mbox{open set},\; A\subset U\}.
\end{array}
$$

\begin{defn}
 A \emph{Haar measure} $\mu_G$ on $G$ is a non-negative regular measure $\mu_G$ on the $\sigma$-algebra
$\mathcal{A}$ such that:
\begin{itemize}
\item $\mu_G (K)<\infty$ for each compact set $K\in\mathcal{A}$;
\item $\mu_G (A)\neq 0$ for each non-empty open set $A\in\mathcal{A}$;
\item $\mu_G (gA)=\mu_G (A)$ $\forall g\in G$ and $\forall A \in \mathcal{A}$,
namely the measure is left invariant.
\end{itemize}
\end{defn}

It is well-known that every locally compact Hausdorff group admits a Haar measure, which
is unique  up to multiplication by a positive costant \cite{Johnson, Weil}.
The group $G$ is called \emph{unimodular} if each left invariant Haar measure on it is also right invariant.
 
From now on we denote by $G$  
the group ${\rm Isom}^+ (\matH^n)$
of orientation-preserving isometries of $\matH^n$, endowed with the compact-open 
topology.

\begin{prop}[\cite{BePe, Ratcliffe}]\label{Haarprop}
The group $G$ is locally compact and unimodular.
Moreover, the Haar measure $\mu_G$ on $G$ can be normalized in such a way that the following condition
holds:
for every basepoint $x\in \matH^n$ and every  Borelian set $R\subseteq \matH^n$, if
$$
S=\{g\in G\, |\, g(x)\in R\}
$$
then $S$ is Borelian  and $\mu_G (S)$ is equal to the hyperbolic volume of $R$.
\end{prop}

From now on we fix a Haar measure $\mu_G$ on $G$ satisfying the
normalization condition described in Proposition~\ref{Haarprop}. 
Keeping notations from the preceding section, 
$\Gamma$ acts properly discontinuously on $G$ via left translations
as a group of measure-preserving diffeomorphisms, so if $W$ is a Borelian subset of $\Gamma\backslash G$
we can set 
$$
\mu_{\Gamma\backslash G} (W)=\mu_G (\widetilde{W}),
$$
where $\widetilde{W}\subseteq G$ 
is any Borelian set that projects bijectively onto $W$.
It is readily seen that this definition of $\mu_{\Gamma\backslash G}(W)$ 
does not depend on the choice of $\widetilde{W}$, and that $\mu_{\Gamma\backslash G}$ is in fact a regular
\emph{right-invariant} measure on $\Gamma\backslash G$.
The following lemma will prove useful later:

\begin{lemma}\label{meas:lemma}
Let $A\subseteq \mhat$ be a Borelian subset, fix a basepoint $x\in\matH^n$ and let
$$
T=\{g\in G\, |\, \pi(g(x))\in A\}.
$$
Then $T$ is a Borelian subset of $G$ such that $\gamma\cdot T=T$ for every $\gamma\in\Gamma$ and
$\mu_{\Gamma \backslash G} (\Gamma \backslash T)=\vol (A)$, 
where $\vol$ denotes the measure induced by the hyperbolic
volume form of $\mhat$. 
\end{lemma}
\begin{proof}
Since the map $f_x\colon G\to \mhat$ defined by $f_x (g)=\pi (g(x))$ is continuous,
if $A$ is Borelian then $T=f_x^{-1} (A)$ is Borelian and $\Gamma$-invariant. Moreover, 
if $D\subseteq \matH^n$ is a Borelian set of representatives for the action of $\Gamma$
on $\matH^n$, then the set $T'=\{g\in G\, |\, g(x)\in D\cap \pi^{-1} (A)\}$ is a Borelian
set of representatives for the left action of $\Gamma$ on $T$. We have therefore 
$\mu_{\Gamma \backslash G} (\Gamma\backslash T)=\mu_G (T')$. On the other hand the restriction
$\pi|_D\colon D\to\mhat$ is measure-preserving, so $\vol (A)=\vol (D\cap \pi^{-1} (A))=
\mu_G (T')$, where the last equality is due to the chosen normalization of 
$\mu_G$.  
\end{proof}

\subsection{$\Gamma$-nets}
In order to define a ``discrete'' smearing procedure in the spirit of~\cite{Loh-Sauer}
we now need to introduce the notion of $\Gamma$-net.
A $\Gamma$-\emph{net} in $\matH^n$ is given by a discrete subset $\lamtil\subseteq \matH^n$ 
(called set of \emph{vertices}) and a collection
of Borelian sets $\{\widetilde{B}_x\}_{x\in \lamtil}$ (called \emph{cells})
such that the following conditions hold:

\begin{enumerate}
\item
$x\in \widetilde{B}_x$ for every $x\in \lamtil$, 
$\matH^n=\bigcup_{x\in\lamtil} \widetilde{B}_x$ 
and $\widetilde{B}_x\cap \widetilde{B}_y=\emptyset$ for every $x,y\in \lamtil$
with $x\neq y$;
\item
$\gamma (\lamtil)=\lamtil$ for every $\gamma\in \Gamma$
and 
$\gamma (\widetilde{B}_x)=\widetilde{B}_{\gamma(x)}$ for every $x\in\lamtil$, $\gamma\in \Gamma$;
\item
${\rm diam}\, (\widetilde{B}_x) \leq 1$ for every $x\in\lamtil$;
\item
if $\widetilde{K}$ is a connected component of $\ext (\widetilde{M}_\vare)$, then
$\widetilde{K}\subseteq \bigcup_{x\in\lamtil\cap \widetilde{K}} \widetilde{B}_x$.
\end{enumerate}

We begin with the following:

\begin{lemma}\label{net:lemma}
There exists a $\Gamma$-net.
\end{lemma}
\begin{proof}
Let $\{T_i\}_{i\in\matN}$ be a smooth triangulation of $\mhat$ 
which restricts to a smooth triangulation of $M_\vare$ and is such that 
${\rm diam} (T_i)\leq 1/2$ for every $i\in\matN$. 
Since $M_\vare$ is compact, the set of simplices of the triangulation whose internal part
is contained in $M_\vare$ is finite, so we may assume that indices are ordered in such a way
that $i<j$ for every pair of simplices such that
${\rm int} (T_i)\subseteq M_\vare$ and ${\rm int} (T_j)\nsubseteq M_\vare$.

If $x_i$ is any point in ${\rm int} (T_i)$ we now set $\lam=\{x_i\}_{i\in\matN}$
and $B_{x_i}=T_i\setminus (\bigcup_{j>i} T_j)$. By contruction $B_{x_i}$
is a simply connected Borelian subset of $\mhat$ for every $i\in\matN$, and 
if $K$ is a connected component of ${\rm ext}({M}_\vare)$, then
$K = \bigcup_{x\in\lam\cap K} B_x$.

Let now $\lamtil=\pi^{-1} (\lam)$. For $i\in\matN$, since $B_{x_i}$ is simply connected
every point $\widetilde{x}\in\pi^{-1} (x_i)$ is contained in exactly one
connected component $\widetilde{B}_{\widetilde{x}}$
of $\pi^{-1} (B_{x_i})\subseteq \matH^n$.

It is now readily seen that the pair $\left(\lamtil, \{\widetilde{B}_x\}_{x\in\lamtil}\right)$
is a $\Gamma$-net.
\end{proof}

\subsection{Smearing}\label{smearing:sub}
Let now $L>0$ be fixed, and let $q^+_0,\ldots q^+_n\in\matH^n$ be the vertices of a fixed regular
simplex with edgelength $L$ such that the embedding $\tau^+_L=[q^+_0,\ldots,q^+_n]$
of the standard simplex in $\matH^n$ is orientation-preserving. 
We also fix an orientation-reversing isometry $g_-$ of $\matH^n$
and set $q_i^- = g_- (q_i)$, 
$\tau_L^- =[q_0^-,\ldots, q_n^-]=g_-\circ \tau^+_L$.

 Let us fix a $\Gamma$-net $\left(\lamtil,\{\widetilde{B}_x\}_{x\in\lamtil}\right)$ and let
$S_n^{\lam} (\mhat)$ be the set of straight simplices in $\mhat$ with vertices in
$\lam=\pi (\lamtil)$.
We would like to define
real coefficients $a_\sigma$ in such a way that the sum
$$
\sum_{\sigma\in S_n^{\lam} (\widehat{M})} a_\sigma \sigma
$$
is finite and defines a cycle in $C_n (\widehat{M},{\rm ext} (M_\vare))$.
Roughly speaking, the coefficient $a_\sigma$ will measure the difference between the accuracy with which 
$\sigma$ approximates an isometric copy of $\tau^+_L$ and 
the accuracy with which 
it approximates an isometric copy of $\tau^-_L$.

So let us fix $\sigma\in S_n^\Lambda (\widehat{M})$, let $\widetilde{\sigma}$ be a lift of
$\sigma$ to $\matH^n$ and let $\widetilde{x}_{i_0},\ldots, \widetilde{x}_{i_n}$ be the vertices of
$\widetilde{\sigma}$. Since $\sigma\in S_n^\lam$ we have $\widetilde{x}_{i_j}\in \lamtil$ for every $j$,
and we denote by $\widetilde{B}_{i_j}$ the cell of the net containing $\widetilde{x}_{i_j}$.
Since different lifts of $\sigma$ differ by the action of an element of $\Gamma$ 
it is easily seen that the sets
$$
\begin{array}{lll}
\Omega_{\sigma}^+ &=&\{[g]\in \Gamma \backslash G\, |\, g (q^+_j) \in \widetilde{B}_{i_j}\ 
{\rm for\ every}\ j \},\\
\Omega_{\sigma}^- &=&\{[g]\in \Gamma \backslash G\, |\, g (q^-_j) \in \widetilde{B}_{i_j}\ 
{\rm for\ every}\ j \}
\end{array}
$$
are well-defined (\emph{i.e.}~do not depend on the chosen lift $\widetilde{\sigma}$)
and Borelian.

We now divide the simplices of $S_n^\lam (\mhat)$ into different classes.
We denote by $W^+$ (resp.~$W^-$)
the set of simplices which intersect ${\rm int}(M_\vare)$ and are ``almost isometric''
to $\tau_L^+$ (resp.~$\tau_L^-$): 
$$
\begin{array}{lll}
W^+ &=& \{\sigma\in S_n^\Lambda (\mhat)\, |\, {\rm Im}\, 
(\sigma) \cap {\rm int}(M_\vare) \neq\emptyset\ {\rm and}\ \Omega_\sigma^+ \neq \emptyset\},\\
W^- &=& \{\sigma\in S_n^\Lambda (\mhat)\, |\, {\rm Im}\, 
(\sigma) \cap {\rm int}(M_\vare) \neq\emptyset\ {\rm and}\ \Omega_\sigma^- \neq \emptyset\}.
\end{array}
$$
Moreover, we denote by $W_{\rm int}^\pm$ the subset of $W^\pm$ given by those simplices
whose image is entirely contained in ${\rm int} (M_\vare)$:
$$
W_{\rm int}^+ = \{\sigma\in W^+\, | \, {\rm Im}\, 
(\sigma)\subseteq {\rm int}(M_\vare)\},\quad 
W_{\rm int}^-  =  \{\sigma\in W^-\, | \, {\rm Im}\, 
(\sigma)\subseteq {\rm int}(M_\vare)\},
$$
and we finally set
$$
W  =  W^+ \cup W^-,\quad  W_{\rm int}=W^+_{\rm int} \cup W^-_{\rm int},\quad W_{\ext} =W\setminus W_{\rm int}.
$$

For every $\sigma \in W$ we now set 
$$
b^+_\sigma= \mu_{\Gamma \backslash G} (\Omega_{\sigma}^+),\quad 
b^-_\sigma= \mu_{\Gamma \backslash G} (\Omega_{{\sigma}}^-),\quad a_\sigma= \frac{b_\sigma^+-b_\sigma^-}{2}.
$$ 

We will prove soon that $W$ is finite and that each $a_\sigma$ is a well-defined real number
(see Lemma~\ref{b:lemma}).
Moreover, we will show in Proposition~\ref{ciclo:prop}
that the sum $\sum_{\sigma\in W} a_\sigma \sigma$ defines a relative cycle 
in $C_n (\mhat, \ext (M_\vare))$, which defines in turn 
a cycle in $C_n (M_\vare,\bb M_\vare)$ via a projection which does not affect simplices
supported in ${\rm int} (M_\vare)$. Therefore if $L$ is large and ``most'' simplices
of $W$ are contained in ${\rm int} (M_\vare)$, then a fundamental cycle for $M_\vare$
exists most simplices of which have volume close to $v_n$. As explained in 
Subsection~\ref{strategy}, this is sufficient for proving that
the simplicial
volume of $M_\vare$ (whence of $\overline{M}$) is close to $\vol (M_\vare)/v_n$ 
(whence to $\vol (M)/v_n$).

\begin{lemma}\label{tetra:lemma}
If $\widetilde{\sigma}\in S_n (\matH^n)$ is a lift of a simplex $\sigma\in W$, 
then ${\rm diam} ({\rm Im}(\widetilde{\sigma}))\leq L+2$.
\end{lemma}
\begin{proof}
The edges of $\tau^\pm_L$ have length $L$, and if $\sigma\in W$ then
the vertices of $\widetilde{\sigma}$ are at distance at most 1 from the vertices
of an isometric copy of $\tau^\pm_L$. This implies that the distance between
any two vertices of $\widetilde{\sigma}$ is at most $L+2$. 
But $\widetilde{\sigma}$ is the convex hull of its vertices and the hyperbolic distance
is convex, so the diameter of $\widetilde{\sigma}$ is 
realized by the distance between two vertices.
\end{proof}

We now fix some notations:
for $Y\subseteq \mhat$ (resp.~$Y\subseteq \matH^n$)
and $R\geq 0$ we denote by $N_R (Y)$ the \emph{closed} $R$-neighbourhood
of $Y$ in $\mhat$ (resp.~in $\matH^n$).

\begin{lemma}\label{b:lemma}
The set $W$ is finite and 
$$
\begin{array}{lllllll}
\vol (M_\vare\setminus N_{L+3} (\bb M_\vare)) &\leq & \sum_{\sigma\in W_{\rm int}^+} b^+_\sigma &\leq & 
\sum_{\sigma\in W^+} b^+_\sigma &\leq &\vol (N_L (M_\vare)),\\
\vol (M_\vare\setminus N_{L+3} (\bb M_\vare)) &\leq & \sum_{\sigma\in W_{\rm int}^-} b^-_\sigma &\leq & 
\sum_{\sigma\in W^-} b^-_\sigma &\leq &\vol (N_L (M_\vare)).
\end{array}
$$
\end{lemma}
\begin{proof}
Let $D\subseteq \matH^n$ be a compact fundamental region for the action of $\Gamma$ on $\widetilde{M}_\vare$.
If $\sigma\in W$ there exists a lift $\widetilde{\sigma}\in S_n (\matH^n)$ 
of $\sigma$ such that ${\rm Im} (\widetilde{\sigma})
\cap D\neq \emptyset$. By Lemma~\ref{tetra:lemma} 
the diameter $\widetilde{\sigma}$
is at most $L+2$, so
${\rm Im} (\widetilde{\sigma})\subseteq N_{L+2} (D)$. However, $\widetilde{\Lambda}$ is discrete and
$N_{L+2} (D)$ is compact, so the number of straight simplices in $\matH^n$ which are contained
in $N_{L+2} (D)$ and have vertices in $\lamtil$ is finite.  
Since every $\sigma \in W$ is obtained by composing such a simplex 
with the covering projection $\pi$ we get that also $W$ is finite.

We now prove the first sequence of inequalities of the statement, the proof of the second one
being very similar.
We define subsets of hyperbolic isometries $H^+_L, K^+_L\subseteq G$ as follows: 
$$
\begin{array}{lll}
K^+_L &=& \{g\in G\, |\, g (q^+_0)\in \widetilde{M}_\vare\setminus N_{L+3} (\bb\widetilde{M}_\vare)\},\\
H^+_L &=& \{g\in G\, |\, g (q_0^+) \in N_L (\widetilde{M}_\vare)\}.
\end{array}
$$
By Lemma~\ref{meas:lemma} such sets
are $\Gamma$-invariant Borelian subsets of $G$ such that
\begin{equation}\label{volumi}
\mu_{\Gamma \backslash G} (\Gamma \backslash K^+_L) = 
\vol ({M}_\vare\setminus N_{L+3} (\bb {M}_\vare)),\qquad
\mu_{\Gamma \backslash G} (\Gamma \backslash H^+_L) = 
\vol (N_L ({M}_\vare)).
\end{equation}

We now show that the following inclusions hold:
\begin{equation}\label{inclusioni}
\Gamma \backslash K^+_L\subseteq 
\bigcup_{\sigma\in W_{\rm int}^+} \Omega^+_\sigma\subseteq 
\bigcup_{\sigma\in W^+} \Omega^+_\sigma\subseteq \Gamma \backslash H^+_L.
\end{equation}
Suppose in fact that $g\in K^+_L$. Then $g(q^+_0)$ lies in 
$\widetilde{M}_\vare\setminus N_{L+3} (\bb\widetilde{M}_\vare)$, 
so if $y_i \in\lamtil$ is the vertex of the cell containing
 $g (q^+_i)$,
then 
$y_0$ lies in $\widetilde{M}_\vare\setminus N_{L+2} (\bb\widetilde{M}_\vare)$.
Let now $\widetilde{\sigma}$ be the straight simplex with vertices $y_0,\ldots,y_n$ and set
$\sigma=\pi\circ\widetilde{\sigma}$. By Lemma~\ref{tetra:lemma} the image 
of $\widetilde{\sigma}$ lies in ${\rm int} (\widetilde{M}_\vare)$.
Since by construction $[g]\in \Omega_\sigma^+ \neq\emptyset$, this implies
$\sigma \in W^+_{\rm int}$.
We have thus proved that $\Gamma \backslash K^+_L\subseteq 
\bigcup_{\sigma\in W_{\rm int}^+} \Omega^+_\sigma$. 

On the other hand suppose $[g]\in \Omega^+_\sigma$ for some $\sigma\in W^+$.
Then there exists a lift $\widetilde{\sigma}$ of $\sigma$ with vertices 
$y_0,\ldots,y_n\in\lamtil$ such that ${\rm Im} (\widetilde{\sigma})\cap {\rm int} (M_\vare)\neq
\emptyset$ and $g(q^+_i)\in \widetilde{B}_{y_i}$ for every $i=0,\ldots,n$. 
Suppose by contradiction $g(q_0^+)\notin N_L (\widetilde{M}_\vare)$. 
Having diameter $L$ and being connected, the image of $g\circ \tau^+_L$  
is then contained in a component $\widetilde{K}$ of $\ext (\widetilde{M}_\vare)$. 
Since $g(q^+_i)\in \widetilde{B}_{y_i}$,
by property~(4) in the definition of $\Gamma$-net we have therefore $y_i\in \widetilde{K}$
for every $i$, so ${\rm Im} (\widetilde{\sigma})\subseteq \widetilde{K}$ by convexity
of $\widetilde{K}$, a contradiction. We have thus proved that $\bigcup_{\sigma\in W^+} \Omega^+_\sigma
\subseteq \Gamma\backslash H^+_L$.

Since the $\Omega^+_\sigma$'s are mutually
disjoint, condition~(\ref{inclusioni}) now implies
$$
\mu_{\Gamma\backslash G} (\Gamma \backslash K^+_L) \leq 
\sum_{\sigma\in W_{\rm int}^+} \mu_{\Gamma\backslash G} (\Omega^+_\sigma) \leq 
\sum_{\sigma\in W^+} \mu_{\Gamma\backslash G} (\Omega^+_\sigma) \leq
\mu_{\Gamma\backslash G} (\Gamma \backslash H^+_L).
$$
By equations~(\ref{volumi}) and the very definition of the $b^+_\sigma$'s,
these inequalities are in fact equivalent to the first sequence of inequalities in the statement. 
\end{proof}

We now set
$$
\zeta_{L,\vare}=\sum_{\sigma\in W} a_\sigma \sigma \in C_n (\mhat).
$$ 

\begin{prop}\label{ciclo:prop}
We have $d\zeta_{L,\vare} \in C_{n-1} ({\rm ext}(M_\vare))$,
so $\zeta_{L,\vare}$ is a relative $n$-cycle in 
$C_\ast (\widehat{M},{\rm ext}(M_\vare))$.
\end{prop}
\begin{proof}
Fix a $(n-1)$-face $\nu\in S_{n-1} (\widehat{M})$ of some $\sigma \in W$.
We will show 
that if
${\rm Im} (\nu)\cap {\rm int} (M_\vare)\neq\emptyset$,
then the coefficient of $\nu$ in $d \zeta_{L,\vare}$ is null.
For every $j=0,\ldots,n$ let us set
$\theta^\pm_j(\nu)=\{ \sigma \in W^\pm\, |\, \bb_j\sigma=\nu \}$
and $\theta_j (\nu)=\theta^+_j (\nu)\cup \theta^-_j (\nu)$.
Since the coefficient of $\nu$ in $d \zeta_{L,\vare}$
is given by
$$
\sum_{j=0}^n (-1)^j\left(\sum_{\sigma \in \theta_j(\nu)} a_{\sigma}\right),
$$
it is sufficient to prove that
under the assumption
${\rm Im} (\nu)\cap {\rm int} (M_\vare)\neq\emptyset$ we have
$$
\sum_{\sigma \in \theta^+_j(\nu)} b^+_{\sigma}=
\sum_{\sigma \in \theta^-_j(\nu)} b^-_{\sigma}
$$
for every $j=0,\ldots,n$.

Let us suppose $j=n$, the other cases being similar.
Let $\widetilde{\nu}$  be a fixed lift of $\nu$ to $\matH^n$, and for every
$\sigma\in \theta_n (\nu)$ let us denote by $\widetilde{\sigma}$  the unique lift of $\sigma$ such that
$\bb_n \widetilde{\sigma}=\widetilde{\nu}$.
By construction 
we have $\widetilde{\nu}(e_i)\in \widetilde{\Lambda}$
for every $i=0,\ldots,n-1$, and we denote by $\widetilde{B}_i$ the cell 
$\widetilde{B}_{\widetilde{\nu}(e_i)}$ 
containing $\widetilde{\nu}(e_i)$. 
Now if $\sigma \in\theta_n (\nu)$ then 
$\widetilde{\sigma} (e_n)$ belongs to $\lamtil$ and 
$\widetilde{\sigma} (e_i)=\widetilde{\tau} (e_i)$
for every $i=0,\ldots,n-1$. Therefore if
$$
\widetilde{\Omega}^\pm_\sigma=\left\{g\in G\, |
\, g(q_i^\pm)\in \widetilde{B}_i\ {\rm for}\ i=0,\ldots,n-1,\ g(q_n^\pm)\in 
\widetilde{B}_{\widetilde{\sigma} (e_n)} \right\}
$$
we have that $\widetilde{\Omega}^\pm_\sigma$ is Borelian and bijectively
projects onto $\Omega^\pm_\sigma \subseteq \Gamma\backslash G$, so that
$b^\pm_\sigma =\mu_G \left(\widetilde{\Omega}^\pm_\sigma\right)$.
Let now
$$
\widetilde{\Omega}^\pm_{{\nu}}=\{g\in G\, |\, 
g(q_i^\pm)\in \widetilde{B}_i\ {\rm for\ every}\ i=0,\ldots,n-1\}.
$$
We claim that $\widetilde{\Omega}^\pm_{{\nu}}=
\bigcup_{\sigma\in \theta^\pm_n (\nu)} \widetilde{\Omega}^\pm_\sigma$.
The inclusion $\supseteq$ is obvious, while in order to get the inclusion
$\subseteq$ it is sufficient to observe that if $g\in \widetilde{\Omega}^\pm_{{\nu}}$
and $\widetilde{\sigma}$ is the straight simplex with vertices
in the cells containing $g(q^\pm_0),\ldots,g(q^\pm_n)$, then  necessarily 
${\rm Im} (\widetilde{\sigma})\cap {\rm int} (\widetilde{M}_\vare)\neq \emptyset$,
so that $\sigma=\pi\circ \widetilde{\sigma}$ belongs to $\theta^\pm_n (\nu)$
and $g$ belongs to $\widetilde{\Omega}^\pm_\sigma$.

Since the $\widetilde{\Omega}^+_\sigma$'s (resp.~the  $\widetilde{\Omega}^-_\sigma$'s)
are pairwise disjoint we finally get
$$
\begin{array}{lllllll}
\sum_{\sigma \in \theta^+_n(\nu)} b^+_{\sigma} &=&\sum_{\sigma \in \theta^+_n(\nu)} \mu_G 
\left(\widetilde{\Omega}^+_\sigma\right)&=&\mu_G \left( \bigcup_{\sigma \in \theta^+_n(\nu)}
\widetilde{\Omega}^+_\sigma
\right)&=&
\mu_G (\widetilde{\Omega}^+_{{\nu}}),\\
\sum_{\sigma \in \theta^-_n(\nu)} b^-_{\sigma} &=&\sum_{\sigma \in \theta^-_n(\nu)} \mu_G 
\left(\widetilde{\Omega}^-_\sigma\right)&=&\mu_G \left( \bigcup_{\sigma \in \theta^-_n(\nu)}
\widetilde{\Omega}^-_\sigma
\right)&=&
\mu_G (\widetilde{\Omega}^-_{{\nu}}),
\end{array}
$$
so in order to conclude we are left to show that $\mu_G (\widetilde{\Omega}^+_{{\nu}})=
 \mu_G (\widetilde{\Omega}^-_{{\nu}})$.
Recall now that a orientation-reversing isometry $g_-\in {\rm Isom}^- (\matH^n)$ exists such that 
$g_- (q^+_i)=q^-_i$ for every $i=0,\ldots,n$ and let $s_n\in {\rm Isom}^- (\matH^n)$ be the reflection
along the face of $\tau^-_L$ opposite to $q^-_n$. It is readily seen that
$g\in \widetilde{\Omega}^-_{{\nu}}$ if and only if $g s_n  g_- \in
\widetilde{\Omega}^+_{{\nu}}$, so $\widetilde{\Omega}^+_{{\nu}}=
\widetilde{\Omega}^-_{{\nu}}\cdot (s_n g_-)$. 
Now the conclusion follows from the fact that
$s_n \circ g_-\in G$ and $\mu_G$
is right-invariant.
\end{proof}

Let now $p_\ast \colon C_n (\widehat{M},{\rm ext}(M_\vare))\to
C_n (M_\vare,\partial M_\vare)$ be the map induced by the piecewise smooth retraction 
 $p\colon (\widehat{M},{\rm ext}(M_\vare))\to
(M_\vare,\partial M_\vare)$
described in Subsection~\ref{N2:sub}.
The cycle
$$
\xi_{L,\vare} = p_\ast (\zeta_{L,\vare})
$$
is our ``efficient cycle'':
in order to prove Theorem~\ref{main:teo}, in the next section
we
estimate both the $L^1$-norm of $\xi_{L,\vare}$ and the proportionality
factor between the class  of $\xi_{L,\vare}$ in $H_n (M_\vare, \bb M_\vare)$
and the fundamental class 
$[M_\vare]$
of $M_\vare$.

\section{Proof of the main theorem}\label{main:sec}

We begin by estimating the $L^1$-norm of $\xi_{L,\vare}$.

\begin{lemma}\label{l1norm:lemma}
We have
$$
\|\xi_{L,\vare}\|\leq
\vol (N_L (M_\vare)).
$$
\end{lemma}
\begin{proof}
Since $p_\ast \colon C_n (\widehat{M},{\rm ext}(M_\vare))\to
C_n (M_\vare,\partial M_\vare)$ is norm non-increasing we have 
$$
2\|\xi_{L,\vare}\| \leq 2 \| \zeta_{L,\vare}\| 
=\sum_{\sigma \in W} \left| {b^+_\sigma -b^-_\sigma}\right|
\leq \sum_{\sigma\in W^+} b_\sigma^+ +  \sum_{\sigma\in W^-} b_\sigma^-\leq 2
\vol (N_L (M_\vare)),
$$
where the last inequality is due to Lemma~\ref{b:lemma}.
\end{proof}

\subsection{The volume form}
In order to compute the proportionality factor between
$[\xi_{L,\vare}]$ and $[M_\vare]$ we would like to evaluate the Kronecker product of $[\xi_{L,\vare}]$
with the volume coclass of $M_\vare$. 
As usual, we first have to take care of the fact that differential forms can be integrated only on smooth
simplices. So let $S^s_k (\mhat)$ (resp.~$S^s_k (\ext (M_\vare))$) be the set of \emph{smooth}
simplices with values in $\mhat$ (resp.~in $\ext(M_\vare)$),
let
$C^s_k (\mhat)$ (resp.~$C^s_k (\ext( M_\vare))$) be the free $\R$-module
generated by $S^s_k (\mhat)$ (resp.~by $S^s_k (\ext(M_\vare))$) and let
us set $C^s_k (\mhat,\ext( M_\vare))=C^s_k (\mhat)/C^s_k (\ext(M_\vare))$. A standard result of differential
topology (see \emph{e.g.}~\cite{Lee}) ensures that a chain map 
$\sm_\ast \colon C_\ast (\mhat)\to C^s_\ast (\mhat)$ exists such that:
\begin{enumerate}
\item
$\sm_k (\sigma)\in S^s_k (\mhat)$ for every $\sigma \in S_k (\mhat)$ and
$\sm_k (\sigma) \in S^s_k (\ext(M_\vare))$
for every $\sigma \in S_k (\ext( M_\vare))$;
\item
$\sm_\ast$ restricts to the identity of $C_\ast^s (M_\vare)$;
\item
if $j_\ast \colon C_\ast^s (M_\vare, \bb M_\vare)\to C_\ast (M_\vare, \bb M_\vare)$
is induced by the natural inclusion, then $\sm_\ast$ induces a map
$C_\ast (M_\vare, \bb M_\vare)\to C_\ast^s (M_\vare, \bb M_\vare)$, which will still be denoted
by $\sm_\ast$, such that $j_\ast\circ \sm_\ast$ is homotopic to the identity
of $C_\ast (M_\vare, \bb M_\vare)$. 
\end{enumerate}

We fix an orientation on $\widehat{M}$ (whence on $M_\vare$) 
by requiring that the fixed covering $\pi\colon \matH^n\to\mhat$
is orientation-preserving, and we denote by $\omega$ the volume differential form on
${\mhat}$. Since the retraction $p\colon \mhat\to M_\vare$
defined in Subsection~\ref{N2:sub} is piecewise smooth, for every $\sigma\in S_n (\mhat)$
it makes sense to integrate $\omega$ over the composition of $\sm_n (\sigma)$ with $p$.
We then define
$\Omega_{\mhat}\colon C_n (\mhat)\to\R$ as the linear extension of the map
$$
\begin{array}{cll}
S_n (\mhat) & \longrightarrow & \R \\
\sigma  & \longmapsto & \int_{p\circ \sm_n (\sigma)} \omega.
\end{array}
$$

By property~(1) of the smoothing operator,
if ${\rm Im} (\sigma)\subseteq {\rm ext} (M_\vare)$
then ${\rm Im} (\sm_n(\sigma))\subseteq {\rm ext} (M_\vare)$, so
${\rm Im} (p\circ\sm_n(\sigma))\subseteq \bb M_\vare$ and $\Omega_{\mhat} (\sigma)=0$.
Moreover, since $\sm_\ast $ is a chain map, if $c\in C_n (\mhat)$ is a boundary then
$\sm_n (c)$ is the boundary of a smooth $(n+1)$-chain, so $p_\ast (\sm_n (c))$ is the
boundary of a piecewise smooth $(n+1)$-chain and 
as a consequence  of Stokes' Theorem
we have $\Omega_{\mhat} (c)=0$. It follows that $\Omega_{\mhat}$ is a cocycle, and
defines therefore a cohomology class $[\Omega_{\mhat}]\in H^n (\mhat,{\rm ext} (M_\vare))$.

Recall now that for every topological pair $(X,Y)$ there exists a 
well-defined pairing (usually called \emph{Kronecker pairing})
given by
$$
\langle \cdot, \cdot \rangle\colon H^p (X,Y) \times H_p (X,Y)  \to \R,\quad
\langle [\varphi], [c]\rangle =\varphi (c).
$$

Let $\zeta_{L,\vare} \in C_n (\mhat, {\rm ext} (M_\vare))$ be the cycle constructed above such that
$\xi_{L,\vare}=p_\ast (\zeta_{L,\vare})$, and let $i\colon (M_\vare, \bb M_\vare)\to (\mhat,{\rm ext} (M_\vare))$
be the inclusion.

\begin{lemma}\label{duality:lemma}
We have 
$$
\left[ \xi_{L,\vare}\right] = \frac{\Omega_{\mhat} (\zeta_{L,\vare})}{\vol (M_\vare)} [M_\vare].
$$
\end{lemma} 
\begin{proof}
We begin by recalling that 
$$
\left\langle \left[ \Omega_{\mhat}\right], i_\ast \left(\left[M_\vare\right]\right)\right\rangle
=\vol \left(M_\vare\right).
$$ 
In fact, if $\varphi\colon \Delta_n\to M_\vare$ is a positively oriented smooth
embedding, then by the very definitions
we have that $\Omega_{\mhat} (i_\ast (\varphi))$ equals the hyperbolic volume
of ${\rm Im} (\varphi)$. We may now represent the fundamental class
$[M_\vare]\in H_n (M_\vare, \bb M_\vare)$ by a finite sum of positively oriented
embeddings $\varphi_i\colon \Delta_n\to T_i$, $i\in I$, where 
$\{T_i\}_{i\in I}$ is a finite smooth triangulation of $M_\vare$.
So
$$
\left\langle \left[ \Omega_{\mhat}\right], i_\ast \left(\left[M_\vare\right]\right)\right\rangle 
=\Omega_{\mhat} \left(i_\ast \left(\sum_{i\in I} 
\varphi_i \right)\right)=
\sum_{i\in I} \vol (T_i)=\vol \left( M_\vare\right).
$$
Since $H_n (\widehat{M},{\rm ext} (M_\vare))\cong H_n(M_\vare, \bb M_\vare)\cong \R$,
this readily implies 
$$
[\zeta_{L,\vare}]=\frac{\langle [\Omega_{\mhat}], [\zeta_{L,\vare}]\rangle}
{ \vol (M_\vare)} \cdot i_\ast ([M_\vare])= \frac{\Omega_{\mhat} (\zeta_{L,\vare})}
{ \vol (M_\vare)} \cdot i_\ast ([M_\vare]),
$$
whence
$$
[\xi_{L,\vare}]=p_\ast ([\zeta_{L,\vare}])=\frac{\Omega_{\mhat} (\zeta_{L,\vare})}
{\vol (M_\vare)} p_\ast (i_\ast ([M_\vare]))=
\frac{\Omega_{\mhat} (\zeta_{L,\vare})}{\vol (M_\vare)} [M_\vare].
$$
\end{proof}

In order to estimate the proportionality coefficient between
$[\xi_{L,\vare}]$ and $[M_\vare]$ we are therefore left to compute $\Omega_{\mhat} (\zeta_{L,\vare})$.

Let $\tau^\pm_L\subseteq \matH^n$ be the regular $n$-simplex
with edgelength $L$ and vertices $q_0^\pm,\ldots, q_n^\pm$ introduced above.
Let $R^\pm_L$ be the set of all straight $n$-simplices $\sigma$ in $\matH^n$ 
satisfying the following property:
there exists
$g\in {\rm Isom}^+ (\matH^n)$ such that the distance between $g(q^\pm_i)$ and
the $i$-th vertex of $\sigma$ is at most one for every
$i=0,\ldots, n$. 
Let 
$$
V_L=\inf \{\vola (\sigma)\, |\, \sigma \in R^+_L\} = -\sup \{\vola (\sigma)\, |\, \sigma \in R^-_L\},
$$
where $\vola$ is the \emph{signed} volume of $\sigma$, \emph{i.e.}~the value obtained
by integrating $\omega$ on $\sigma$ (so $|\vola (\sigma)|=\vol ({\rm Im}(\sigma))$). 

It is easily seen that $L_0>0$ exists such that for every element $\sigma \in R^+_L$
(resp.~$\sigma\in R^-_L$) we have $\vola (\sigma)>0$ (resp.~$\vola (\sigma)<0$).
Moreover, the hyperbolic volume of a geodesic simplex
with vertices on $\matH^n \cup \bb \matH^n$ is a continuous function of its vertices
(see~\cite[Theorem 11.3.2]{Ratcliffe}). As a consequence we have

\begin{equation}\label{voleq}
\lim\limits_{L\to\infty} V_L=v_n. 
\end{equation}

\begin{prop}\label{stimadiff:prop}
Let $L\geq L_0$. Then
$$
\Omega_{\mhat} (\zeta_{L,\vare}) \geq V_L \cdot \vol (M_\vare\setminus N_{L+3} (\bb M_\vare)).
$$
\end{prop}
\begin{proof}
With notations as in Subsection~\ref{smearing:sub} we 
set
$$
\zeta_{\rm int}^+=\sum_{\sigma\in W_{\rm int}^+} b^+_\sigma \sigma, \quad
\zeta_{\rm int}^-=- \sum_{\sigma\in W_{\rm int}^-} b^-_\sigma \sigma, \quad
\zeta_{\rm ext}=\sum_{\sigma\in W_{\ext}} a_\sigma \sigma.
$$

Take $\sigma \in W$.
By the very definitions $\Omega_{\mhat} (\sigma)$ is the signed volume
of the portion of ${\rm Im} (\sigma)$ contained in $M_\vare$, so if 
$\sigma \in W_{\rm int}^+$ (resp. $\sigma\in  W_{\rm int}^-$) we have 
$$
\Omega_{\mhat} (b^+_\sigma \sigma)=b^+_\sigma \vola (\sigma)\geq b^+_\sigma \cdot V_L,
\qquad 
\Omega_{\mhat} (-b^-_\sigma \sigma)=-b^-_\sigma \vola (\sigma)\geq b^-_\sigma \cdot V_L.
$$
By Lemma~\ref{b:lemma} we now get
\begin{equation}\label{stima:eq}
\begin{array}{lllll}
\Omega_{\mhat} (\zeta^+_{\rm int})&\geq& \left( \sum_{\sigma\in W^+_{\rm int}}  b^+_\sigma \right)
V_L &\geq & \vol (M_\vare\setminus N_{L+3} (M_\vare))\cdot V_L,\\
\Omega_{\mhat} (\zeta^-_{\rm int})&\geq &\left( \sum_{\sigma\in W^-_{\rm int}}  b^-_\sigma \right)
V_L &\geq & \vol (M_\vare\setminus N_{L+3} (M_\vare))\cdot V_L.
\end{array}
\end{equation}

Moreover, since $L\geq L_0$ 
if $b_\sigma^+\neq 0$ then $\sigma$ is positively oriented, while 
if $b_\sigma^-\neq 0$ then $\sigma$ is negatively oriented. This easily implies that
for every
$\sigma\in W_\ext$ we have $\Omega_{\mhat} (a_\sigma \sigma)\geq 0$, so  
$\Omega_{\mhat} (\zeta_\ext)\geq 0$.
Together with inequalities~(\ref{stima:eq}) and the fact that
$\zeta_{L,\vare}=(\zeta^+_{\rm int}+\zeta^-_{\rm int})/2+\zeta_{\rm ext}$, 
this readily implies the conclusion.
\end{proof}

\begin{cor}\label{stimafinale}
We have
\begin{equation}\label{stimafin:eq}
\frac{\vol (M)}{|| \overline{M}||}\geq \frac{ 
\vol (M_\vare\setminus N_{L+3}(\bb M_\vare))}{\vol (N_L(M_\vare))}\cdot V_L
\end{equation}
\end{cor}
\begin{proof}
Since $\overline{M}$ is diffeomorphic to $M_\vare$
we have $\| \overline{M}\|=\| M_\vare\|=\| [M_\vare]\|$. 
Thus from Lemma~\ref{l1norm:lemma}, Lemma~\ref{duality:lemma}
and Proposition~\ref{stimadiff:prop} we get
$$
{|| \overline{M}||}=
\left\|\frac{\vol(M_\vare)}{\langle [\Omega_{\mhat}],\xi_{L,\vare}\rangle } 
\left[ \xi_{L,\vare}\right]\right\|
\leq \left| \frac{\vol(M_\vare)}{\langle [\Omega_{\mhat}],\xi_{L,\vare}\rangle }\right| \|\xi_{L,\vare}\|
\leq
\frac{\vol (N_L(M_\vare)) \cdot \vol (M_\vare)}{V_L \cdot 
\vol (M_\vare\setminus N_{L+3}(\bb M_\vare))}.
$$
Since $\vol (M)\geq \vol (M_\vare)$, this readily implies the conclusion.
\end{proof}

\subsection{The final step}
In order to conclude we now need some estimates on the volume of 
$L$-neighbourhoods of geodesic hypersurfaces in hyperbolic
manifolds.
For $t\geq 0$ let $g(t)= 2\int_0^L \cosh^{n-1} (t)\, {\rm d} t$.
An easy computation (see \emph{e.g.}~\cite{Ara}) shows 
that if $A$ is an embedded totally geodesic hypersurface
in a hyperbolic $n$-manifold $X$, then the $n$-dimensional volume 
of any 
embedded tubular $t$-neighbourhood of $A$ in $X$ is given by
$g(t)\cdot \vol (A)$.

\begin{lemma}\label{vol:lemma}
For every $L>0$ we have
$$
\lim_{\vare \to 0} {\vol (N_L (\bb M_\vare))}\leq 
g(L)\cdot {\vol (\bb M)}.
$$ 
\end{lemma}
\begin{proof}
Recall that $\bb M_\vare=\bb M \cup T_\vare$, 
where $T_\vare$ is the union of the boundaries of the deleted cusps.
Therefore $N_L( \bb M_\vare)=N_L (\bb M)\cup N_L (T_\vare)$ and it is easily seen that 
$\lim_{\vare \to 0} \vol (N_L (T_\vare))=0$ whence 
$\lim _{\vare \to 0} {\vol (N_L (\bb M_\vare))}=
{\vol (N_L (\bb M))}$.

Let now $B$ be a connected component of $\bb M$ and
let $X\to \widehat{M}$ be the Riemannian covering associated to the image of
$\pi_1 (B)$ into $\pi_1 (M)$. Then $X$ is diffeomorphic to $B\times (-\infty,+\infty)$
and contains a totally geodesic hypersurface $B\times \{0\}$ isometric to $B$. The $L$-neighbourhood
of $B\times \{0\}$ in $X$ is embedded and has therefore volume $g(L)\cdot \vol (B)$. Since
the projection $X\to \mhat$ is a local isometry and maps (possibly not injectively)
such a neighbourhood onto
$N_L (B)\subseteq \mhat$, it follows that $\vol (N_L (B))\leq g(L)\cdot \vol (B)$.
If $B_1,\ldots, B_k$ are the components of $\bb M$ we then have 
$$
\vol (N_L (\bb M))\leq
\sum_{i=1}^k \vol (N_L (B_i))
\leq g(L)\left( \sum_{i=1}^k \vol (B_i) \right)
=g(L)\cdot \vol (\bb M),
$$
whence the conclusion.
\end{proof}

Let us put the estimate of Lemma~\ref{vol:lemma} into inequality~(\ref{stimafin:eq}).
We have 
$$
\vol (M_\vare\setminus N_{L+3} (\bb M_\vare))\geq \vol (M_\vare) - 
\vol(N_{L+3}(\bb M_\vare)),
$$ 
so since $\lim_{\vare \to 0} \vol (M_\vare)=\vol (M)$
we get 
$$
\lim_{\vare \to 0}  \vol (M_\vare\setminus N_{L+3} (\bb M_\vare))\geq
\vol (M) - g(L+3)\cdot \vol (\bb M).
$$  
In the same way we get $\lim_{\vare \to 0} \vol (N_L (M_\vare))\leq \vol (M) + g(L)\cdot 
\vol (\bb M)$. 
Therefore, if $r=\vol (\bb M)/\vol (M)$, then
passing to the limit in the right hand side of~(\ref{stimafin:eq}) we obtain
\begin{equation}\label{last}
\frac{\vol (M)}{|| \overline{M}||}\geq \frac{ 1-r\cdot  g(L+3)}
{1+r \cdot g(L)}\cdot V_L.
\end{equation}
Let now $\eta<v_n$ be given. By equation~(\ref{voleq}) there exists $L_1\geq L_0$ (only depending on $n$ and
$\eta$) such that
$V_{L_1}> v_n-\eta/2$. Since $\lim_{r\to 0} (1-r\cdot g(L_1+3))/(1+r \cdot g(L_1))=1$, there exists $k>0$
(only depending on $L_1$, that is on $n$ and $\eta$) 
such that $(1-r \cdot g(L_1+3))/(1+r \cdot g(L_1))>(v_n-\eta)/(v_n-\eta/2)$ for every $r\leq k$.
Inequality~(\ref{last}) with $L=L_1$ now shows that if $r=\vol (\bb M)/\vol (M)\leq k$ then
$$
\frac{\vol (M)}{|| \overline{M}||}\geq \frac{v_n-\eta}{v_n-\eta/2}(v_n-\eta/2)=v_n -\eta,
$$
and this concludes the proof of Theorem~\ref{main:teo}.

\bibliographystyle{amsalpha}
\bibliography{biblio}

\end{document}